\documentclass[a4paper,12pt,reqno]{amsart}
\usepackage[body={17cm,21cm}]{geometry}
\usepackage[initials]{amsrefs}

\usepackage{amssymb,amscd}
\usepackage[all]{xy}

\usepackage[mathscr]{eucal}
\usepackage{enumerate}

\numberwithin{equation}{section}
\theoremstyle{plain}
\newtheorem{thm}[equation]{Theorem}
\newtheorem{cor}[equation]{Corollary}
\newtheorem{lem}[equation]{Lemma}
\newtheorem{prop}[equation]{Proposition}

\newtheorem{definition}[equation]{Definition}
\newtheorem{exa}[equation]{Example}
\newtheorem{rem}[equation]{Remark}

\theoremstyle{definition}

\theoremstyle{remark}

\DeclareMathOperator{\Br}{Br}

\DeclareMathOperator{\codim}{codim}

\DeclareMathOperator{\Div}{Div}

\DeclareMathOperator{\Gal}{Gal}

\DeclareMathOperator{\Hom}{Hom}

\DeclareMathOperator{\Pic}{Pic}

\DeclareMathOperator{\Res}{Res}

\DeclareMathOperator{\Spec}{Spec}

\def\Br{{\rm Br\,}}
\def\inv{{\rm inv\,}}

\def\Gal{{\rm Gal}}

\def\ker {{\rm  Ker}}

\def\Pic{{\rm Pic}}
\def\ker{{\rm ker}}

\DeclareFontEncoding{OT2}{}{} 
   \newcommand{\textcyr}[1]{%
     {\fontencoding{OT2}\fontfamily{wncyr}\fontseries{m}\fontshape{n}%
      \selectfont #1}}

\newcommand{\Sha}{{\mbox{\textcyr{Sh}}}}

\DeclareFontFamily{U}{wncy}{}
\DeclareFontShape{U}{wncy}{m}{n}{%
<5>wncyr5%
<6>wncyr6%
<7>wncyr7%
<8>wncyr8%
<9>wncyr9%
<10>wncyr10%
<11>wncyr10%
<12>wncyr6%
<14>wncyr7%
<17>wncyr8%
<20>wncyr10%
<25>wncyr10}{}
\DeclareMathAlphabet{\cyr}{U}{wncy}{m}{n}
\begin{document}

\title[]
{Strong Approximation with Brauer-Manin Obstruction for Toric Varieties}

\author{Yang CAO}

\address{Yang CAO \newline School of Mathematical Sciences, \newline Capital Normal University,
\newline 105 Xisanhuanbeilu, \newline 100048 Beijing, China}

\email{yangcao1988@gmail.com}

\author{Fei XU}

\address{Fei XU \newline School of Mathematical Sciences, \newline Capital Normal University,
\newline 105 Xisanhuanbeilu, \newline 100048 Beijing, China}

\email{xufei@math.ac.cn}


\date{\today.}




\maketitle

\begin{abstract} For smooth open toric varieties, we establish strong approximation off infinity with Brauer-Manin obstruction.
\end{abstract}

\tableofcontents

\section{Introduction}

Strong approximation has various arithmetic application, for example to determine the existence of integral points by the local-global principle. By using Manin's idea, J.-L. Colliot-Th\'el\`ene and F. Xu established strong approximation with Brauer-Manin obstruction for homogeneous spaces of semi-simple and simply connected algebraic groups in \cite{CTX} to refine the classical strong approximation. Since then, a significant progress for strong approximation with Brauer-Manin obstruction has been made for various homogeneous spaces of linear algebraic groups in \cite{Ha08}, \cite{D}, \cite{WX}, \cite{BD} and families of homogeneous spaces in \cite{CTX1}, \cite{CTH}. In this paper, we study strong approximation with Brauer-Manin obstruction for open smooth toric varieties. Such varieties have been extensively studied over algebraic closed fields (see \cite{Fulton} and \cite{Oda}). However they are hard to study over number fields. For example, a smooth toric variety may not have an open affine toric subvariety covering over a field.

Notation and terminology are standard. Let $k$ be a number field, $\Omega_k$ be the set of all primes in $k$ and  $\infty_k$ be the set of all archimedean primes in $k$. Write $v<\infty_k$ for $v\in \Omega_k\setminus \infty_k$. Let $O_k$ be the ring of integers of $k$ and $O_{k,S}$ be the $S$-integers of $k$ for a finite set $S$ of $\Omega_k$ containing $\infty_k$. For each $v\in \Omega_k$, the completion of $k$ at $v$ is denoted by $k_v$ and the completion of $O_k$ at $v$ by $O_v$. Write $O_{v}=k_v$ for $v\in \infty_k$. Let ${\mathbf A}_k$ be the adelic ring of $k$ and ${\mathbf A}_k^\infty$ be the finite adeles of $k$.

For any scheme $X$ of finite type over $k$, we denote $$\Br(X)=H_{\text{\'et}}^2(X, \Bbb G_m),  \ \ \ \Br_1(X)= \ker[\Br(X)\rightarrow \Br(X_{\bar{k}})], \ \ \ \Br_a(X)=\Br_1(X)/\Br(k)  $$ where $\Bbb G_m$ is a group scheme defined by the multiplicative group and $X_{\bar k}=X\times_k \bar{k}$ with a fixed algebraic closure $\bar{k}$ of $k$. We also use $\Bbb A^n$ to denote an affine space of dimension $n$. For any subset $B$ of $\rm Br(X)$, one defines
$$ X({\mathbf A}_k)^B = \{ (x_v)_{v\in \Omega_k}\in X({\mathbf A}_k): \ \ \sum_{v\in \Omega_k} \inv_v(\xi(x_v))=0, \ \ \forall \xi\in B \}  $$ which is a closed subset of $X(\mathbf A_k)$. As discovered by Manin, class field theory implies that $X(k) \subseteq X(\mathbf A_k)^B$. Let $\rm{Pr_\infty}$ denote the projection from adelic points to finite adelic points.

\begin{definition} Let $X$ be a scheme of finite type over $k$, and $S$ a finite subset of $\Omega_k$.

 i) If $X(k)$ is dense in $X({\mathbf A}_k^S)$, we say $X$ satisfies strong approximation off $S$.

 ii) If $X(k)$ is dense in $ \rm{Pr_S} ( X({\mathbf A}_k)^{\Br(X)})$, we say $X$ satisfies strong approximation with Brauer-Manin obstruction off $S$.
\end{definition}

In this paper, we will study strong approximation for toric varieties defined as follows.

\begin{definition} \label{tv} Let $T$ be a torus over $k$ and $X$ be an integral normal and separated scheme of finite type over $k$ with an action of $T$
$$ m_X: T\times_k X \longrightarrow X  $$ over $k$. An open immersion $i_T : T \hookrightarrow X$ over $k$ is called a toric variety over $k$ if the following diagram commutes
\[ \begin{CD} T\times_k T @>{m_T}>> T \\
 @V{id\times i_T}VV   @VV{i_T}V  \\
 T\times_k X @>>{m_X}> X
\end{CD} \]
where $m_T$ is the multiplication of $T$. We simply write $(T\hookrightarrow X)$ or $X$ for this toric variety if the open immersion is clear.
\end{definition}

The main result of this paper is the following theorem.

\begin{thm} Any smooth toric variety over $k$ satisfies strong approximation with Brauer-Manin obstruction off $\infty_k$. \end{thm}

As a corollary, we have:

\begin{cor}
Let $S$ be a subset of $\Omega_k$, such that $\infty_k\subseteq S$. Then any smooth toric variety over $k$ satisfies strong approximation with Brauer-Manin obstruction off $S$.
\end{cor}


Chambert-Loir and Tschinkel prove the same result in \cite{ChTs} under certain conditions by using harmonic analysis. More precisely, let $(T\hookrightarrow X)$ be a smooth projective toric variety over $k$ and $D$ a $T$-invariant divisor of $X$ with $U=X\setminus D$. Assuming the line bundle $-(K_X+D)$ is big where $K_X$ is a canonical bundle of $X$ and $\Pic(U)$ is free (see the proof of Lemma 3.5.1 in \cite{ChTs} and also Remark \ref{picard}), they establish asymptotic formulas for integral points of $U$, which imply that $U$ satisfies strong approximation with Brauer-Manin obstruction off $\infty_k$.

Also we learned that D. Wei has obtained the same result in \cite{W} under the condition $\bar{k}[X]^\times =\bar{k}^\times$. More precisely, he prove that for any smooth toric variety $X$ satisfying $\bar{k}[X]^\times =\bar{k}^\times$, any closed subset $W\subseteq X$ with ${\rm codim}(W,X)\geq 2$, and any $v_0\in \Omega_k$, the variety $X-W$ satisfies strong approximation with Brauer-Manin obstruction off $v_0$. Without the condition $\bar{k}[X]^\times =\bar{k}^\times$, this result does not hold in general (see Example \ref{contre}).

This paper is organized as follows.

In section 2, we study the structure of smooth toric varieties over an arbitrary field of characteristic 0. We give a structure theorem for affine smooth toric varieties (Proposition \ref{str}). We then defined the notion of smooth toric varieties of pure divisorial type (Definition \ref{d}) and the notion of standard toric varieties (Definition \ref{standard}). In any smooth toric variety, there exists a closed subvariety of codimension $\geq 2$, whose complement is a smooth toric variety of pure divisorial type (Proposition \ref{red}). We construct a morphism from a standard toric variety to a given toric variety, and prove a structure theorem for smooth toric varieties by this morphism (Proposition \ref{rho}).

In section 3, we extend strong approximation with Brauer-Manin obstruction off $\infty_k$ for tori proved by Harari in \cite{Ha08} to a relative strong approximation with Brauer-Manin obstruction off $\infty_k$ for tori (Proposition \ref{relative}). We establish strong approximation off $\infty_k$ for standard toric varieties (Corollary \ref{v}).

In section 4, using the morphism constructed in section 2, we establish the crucial step (Proposition \ref{int}), which gives a precise relation between the $O_v$-points of a given toric variety and the $O_v$-points of a standard toric variety for almost all place $v\in \Omega_k$. Then, by combining relative strong approximation for tori and strong approximation for standard toric varieties, we establish strong approximation with Brauer-Manin obstruction off $\infty_k$ for smooth toric varieties of pure divisorial type (Proposition \ref{typed}), and then for any smooth open toric varieties (Theorem \ref{end}).

In section 5, we give an example (Example \ref{contre}), which shows that the complement of a point in a toric variety may no longer satisfy strong approximation with Brauer-Manin obstruction off $\infty_k$. This is in contrast with the case of affine space minus a closed subscheme of codimension $\geq 2$ (Proposition \ref{cdim}).

\section{Structure of smooth toric varieties}

Toric varieties have been extensively studied over an algebraically closed field (see \cite{Fulton} and \cite{Oda}). In this section, we study the structure of toric varieties over a field $k$ with $char(k)=0$. Let $\bar k$ be an algebraic closure of $k$. For a torus $T$ over $k$, we denote the character group of $T$ by $T^*=Hom_{\bar{k}}(T, \Bbb G_m)$, which is a free $\Bbb Z$-module of finite rank with continuous action of $\Gamma_k=\Gal(\bar{k}/k)$. It is well-known that these two categories are anti-equivalent (see Proposition 1.4 of Expos\'e X in \cite{SGA3}). For convenience, we recall the following definition.

The objects of the category of toric varieties over $k$ are toric embeddings $(T\hookrightarrow X)$ over $k$, and a morphism $(T\hookrightarrow X) \xrightarrow{f} (T'\hookrightarrow X')$ in this category is given by a morphism $f: X\rightarrow X'$ of schemes over $k$ such that the restriction of $f$ to $T$ gives a homomorphism $T\xrightarrow{f|_{T}} T'$ over $k$ and the following diagram
 \[ \begin{CD} T\times_k X  @>{f|_T\times f}>> T'\times_k X' \\
 @V{m_X}VV   @VV{m_{X'}}V  \\
 X @>>{f}> X'
\end{CD}
\]
commutes over $k$.

 If $f: X\rightarrow X'$ is an isomorphism of schemes over $k$ and induces isomorphism $T\cong T'$ of tori over $k$, then $f$ is called an isomorphism of toric varieties $(T\hookrightarrow X)$ and $(T'\hookrightarrow X')$ over $k$. In this case, two toric varieties $(T\hookrightarrow X)$ and $(T'\hookrightarrow X')$ are called isomorphic over $k$.

If $f: X\rightarrow X'$ is a closed immersion over $k$, then $f$ is called a closed immersion of toric varieties over $k$.

If $f: X\rightarrow X'$ is an open immersion and $f_T$ is an isomorphism of tori over $k$, then $X$ is called an open toric subvariety of $X'$ over $k$.

The simplest example of toric variety is $\Bbb A^s \times \Bbb G_m^t$ containing the natural open torus $\Bbb G_m^{s+t}$ for some non-negative integers $s$ and $t$. Such toric varieties are the building blocks of smooth toric varieties. The following lemma is due to Sumihiro in \cite{Su}.

\begin{lem} \label{open}(Sumihiro) Let $k=\bar k$. Any toric variety $(T\hookrightarrow X)$ has a finite open covering $\{U_j\}$ of $X$ over $\bar{k}$ such that all $(T\hookrightarrow U_j)$'s are affine toric sub-varieties over $\bar{k}$. Moreover, if $X$ is smooth, then one has isomorphisms of toric varieties over $\bar k$
\[ \begin{CD} T @>{\cong}>>  \Bbb G_m^{s_j+t_j} \\
   @V{i_T}VV @VVV  \\
U_j @>>{\cong}> \Bbb A^{s_j} \times_{\bar k} \Bbb G_m^{t_j}
\end{CD} \]
with some integers $s_j, t_j\geq 0$  and $s_j+t_j=\dim (T)$ for each $j$, where $i_T$ is the open immersion in Definition \ref{tv}.
\end{lem}

\begin{proof} By Lemma 8 and Corollary 2 in \cite{Su}, one has a finite affine open covering  $\{U_j\}$ of $X$ over $\bar{k}$ such that all $U_j$'s are $T$-stable. Since $X$ is irreducible, one has $U_j\cap i_T(T)\neq \emptyset$ where $i_T$ is the open immersion in Definition \ref{tv}. Take $$x_0=i_T(t_0) \in U_j(\bar{k})\cap i_T(T(\bar{k}))$$ with $t_0\in T(\bar{k})$ and one obtains
$$ i_T(T(\bar{k})) = i_T(T(\bar{k})t_0)\subseteq U_j(\bar{k}) $$ by the commutative diagram in Definition \ref{tv}. Therefore $i_T: T\hookrightarrow U_j$ for all $j$ by Hilbert Nullstellensatz and all $U_j$'s are toric varieties with respect to $T$.

If $X$ is smooth, all $U_j$'s are smooth. Thus $(T\hookrightarrow U_j)$ is isomorphic to $(\Bbb G_m^{s_j+t_j} \hookrightarrow \Bbb A^{s_j} \times_{\bar k} \Bbb G_m^{t_j})$ by the criterion of smoothness for affine toric variety (see Theorem 1.10 in \cite{Oda}) for all $j$.
\end{proof}

\begin{rem} Lemma \ref{open} does not hold over a general field. For example, the conic $x^2-a y^2=z^2$ inside ${\Bbb P}^2$ over $\Bbb Q$ with $a\not\in (\Bbb Q^\times)^2$. This conic is a toric variety containing an open subset with $z\neq 0$ which is isomorphic to the restriction of scalar of the norm one torus
$$T = \Res_{\Bbb Q(\sqrt{a})/\Bbb Q}^1 (\Bbb G_m) . $$ This toric variety has no open affine toric subvariety covering over $\Bbb Q$.
\end{rem}

The set of rational points of toric varieties can be covered by open affine toric sub-varieties.

\begin{cor}\label{affine} Let $(T\hookrightarrow X)$ be a toric variety over $k$. If $x\in X(k)$, there is an open affine toric subvariety $(T\hookrightarrow M)$ of $(T\hookrightarrow X)$ over $k$ such that $x\in M(k)$. \end{cor}

\begin{proof}  For $x\in X(k)$, there is a finite Galois extension $k'/k$ and an open affine toric variety $(T_k\times_k k'\hookrightarrow U)$ over $k'$ such that $x\in U(k')$ by Lemma \ref{open}. Then $$x\in M=\bigcap_{\sigma\in Gal(k'/k)} \sigma (U)$$ and $M$ is stable under $Gal(k'/k)$. One concludes that $M$ is defined over $k$ by Galois descent (see Corollary 1.7.8 in \cite{Fu}) and $(T\hookrightarrow M)$ is an open affine toric variety over $k$ by separateness of $X$.
\end{proof}

\begin{cor} \label{orb} If $(T\hookrightarrow X)$ is a smooth toric variety over $k$, then $X(\bar{k})$ consists of finitely many $T(\bar{k})$-orbits.
\end{cor}

\begin{proof} By Lemma \ref{open}, one only needs to show that $(\bar{k})^s\times (\bar{k}^\times)^t$ with the natural action $(\bar{k}^\times)^{s+t}$ has finitely many orbits. Suppose $(x_i)$ and $(y_i)$ are in $(\bar{k})^s\times (\bar{k}^\times)^t$. Then $(x_i)$ and $(y_i)$ are in the same orbit of $(\bar{k}^\times)^{s+t}$  if and only if for $1\leq i\leq s+t$ either $x_i\cdot y_i \neq 0$ or $x_i=y_i=0$. This implies the finiteness of  $(\bar{k}^\times)^{s+t}$-orbits.
\end{proof}

Since the $\bar k$-orbits are finite for a smooth toric variety, by Galois descent, there is a smallest open affine toric subvariety containing a given rational point over $k$.

\begin{prop}\label{str} If $(T\hookrightarrow X)$ is a smooth affine toric variety over $k$, there is a unique closed toric subvariety
$$(\Res_{K/k}(\Bbb G_m) \hookrightarrow \Res_{K/k}(\Bbb A^1))$$
 of $(T\hookrightarrow X)$ with a finite \'etale $k$-algebra $K/k$ such that the quotient homomorphism
 $$\phi: T\rightarrow T_1 \ \ \ \text{with} \ \ \ T_1=T/\Res_{K/k}(\Bbb G_m)$$
 can be extended to a morphism $\phi: X\rightarrow T_1$ over $k$ commuting with the action
 \[ \begin{CD} T\times_k X  @>{\phi\times \phi}>> T_1\times_k T_1 \\
 @V{m_X}VV   @VV{m_{T_1}}V  \\
 X @>>{\phi}> T_1
\end{CD}
\]
and $\phi^{-1}(1)\cong \Res_{K/k}(\Bbb A^1)$.
Moreover, $\phi$ induces an isomorphism $\Br_1(T_1)\stackrel{\sim}{\rightarrow} \Br_1(X)$.
\end{prop}

\begin{proof} Since $\Pic(X_{\bar{k}})=0$, one has the following short exact sequence
$$ 1\rightarrow \bar{k}[X]^\times /{\bar k}^\times \rightarrow \bar{k}[T]^\times /{\bar k}^\times\rightarrow \Div_{X_{\bar k} \setminus T_{\bar k}}(X_{\bar k}) \rightarrow 1 $$ of $\Gamma_k$-module by sending $f\mapsto div_{X_{\bar{k}}\setminus T_{\bar{k}}}(f)$ for any $f\in  \bar{k}[T]^\times$. There is a finite \'etale $k$-algebra $K/k$ such that
$$ (\Res_{K/k}(\Bbb G_m))^* = \Div_{X_{\bar k} \setminus T_{\bar k}}(X_{\bar k}) \ \ \ \text{and} \ \ \ T_1^*= \bar{k}[X]^\times /{\bar k}^\times . $$

Let $$ B = \{ f\in \bar k[X]^\times: \  f(1_T)=1 \} $$ which is stable under the action of $\Gamma_k$. Then $$\bar{k}[X]^\times \cong \bar k^\times \oplus B, \ \ f\mapsto (f(1), f(1)^{-1}f) $$ as $\Gamma_k$-module. The $\bar k$-algebra isomorphism $$\bar{k}[T_1]\cong \bar{k}[B]  \ \ \ \text{induced by} \ \ \ B\cong \bar{k}[X]^\times /{\bar k}^\times$$ is compatible with $\Gamma_k$-action. Moreover, the natural inclusion of $\bar k$-algebras $\bar{k}[B] \subseteq \bar{k}[X]$ is compatible with $\Gamma_k$-action as well. This gives the morphism $X \rightarrow T_1$ over $k$ which extends $\phi: T\rightarrow T_1$. Since $\phi$ is a homomorphism of tori, this implies that the above diagram commutes.

Choose compatible isomorphisms
$$T_{\bar k} \cong \Spec(\bar k[x_1, x_1^{-1},\cdots, x_{s},x_{s}^{-1}, y_1, y_1^{-1}, \cdots, y_t, y_t^{-1}]) $$ and
 $$ X_{\bar k}\cong \Bbb A^s \times_{\bar k} \Bbb G_m^t= \Spec (\bar{k} [x_1, \cdots, x_{s}, y_1, y_1^{-1}, \cdots, y_t, y_t^{-1}])  $$
 such that $x_i(1_T)=y_j(1_T)=1$ for $1\leq i\leq s$ and $1\leq j\leq t$ by Lemma \ref{open}.  Then
  $$T_1\times_k \bar k = \Spec (\bar{k} [y_1, y_1^{-1}, \cdots, y_t, y_t^{-1}]) \ \ \ \text{and} \ \ \  \bar \phi=\phi\times_k \bar k: \ \  X_{\bar k} \rightarrow T_1\times_k \bar k $$
  is the projection and
 $$ \phi^{-1}(1)\times_k \bar k= \bar{\phi}^{-1}(1)\cong \Spec (\bar{k} [x_1, \cdots, x_{s}]) .$$ Since $div_{X_{\bar{k}}\setminus T_{\bar{k}}}(x_i)=div_{X_{\bar{k}}}(x_i)$ for $1\leq i\leq s$ and the action of $\Gamma_k$ on $\{div_{X_{\bar{k}}\setminus T_{\bar{k}}}(x_i)\}_{i=1}^s $ is given by permutation, one concludes that $\Gamma_k$ acts on the coordinates $\{x_i\}_{i=1}^s$ by permutation by smoothness of $X$. This implies that
$ \phi^{-1}(1)\cong \Res_{K/k}(\Bbb A^1)  $ as required. Moreover, $\phi: X\rightarrow T_1$ is faithfully flat, since $\bar \phi=\phi\times_k \bar k$ is a projection.

Now we prove the uniqueness. Suppose that $(T\hookrightarrow X)$ contains another closed toric subvariety
$$(\Res_{K'/k}(\Bbb G_m) \hookrightarrow \Res_{K'/k}(\Bbb A^1))$$
with a finite \'etale $k$-algebra $K'/k$ such that the quotient homomorphism
$$\phi': T\rightarrow T'_1 \ \ \ \text{with} \ \ \ T'_1=T/\Res_{K'/k}(\Bbb G_m)$$
can be extended to a morphism $\phi': X\rightarrow T'_1$ over $k$ satisfying $\phi'^{-1}(1)=\Res_{K'/k}(\Bbb A^1)$. In this case, $\phi'$ induces an injective $\Gamma_k$-homomorphism
$$\chi^*: \ T_1'^*\rightarrow \bar k[X]^\times/\bar k^\times = T_1^*  \ \ \ \text{such that} \ \ \ T_1^*/\chi^*(T_1'^*) \ \text{is torsion free} $$
and $\phi'=\chi\circ \phi$ with $T_1\xrightarrow{\chi} T_1'$ is induced by $\chi^*$. Since $\phi'^{-1}(1)=\Res_{K'/k}(\Bbb A^1)$, one has
$$\bar k[\phi'^{-1}(1)]^\times =\bar k^\times .$$
Since $\phi: X\rightarrow T_1$ is faithfully flat, $\phi: \phi'^{-1}(1)\rightarrow \chi^{-1}(1)$ is faithfully flat. Thus $\phi^*: \bar k[\chi^{-1}(1)]^\times\rightarrow \bar k[\phi'^{-1}(1)]^\times=\bar k^\times$ is injective. Since $\chi^{-1}(1)={\rm ker}(\chi)$, $\bar k[{\rm ker}(\chi)]^\times=\bar k^\times$, and ${\rm ker}(\chi)$ is trivial. This implies that $T_1^*=\chi^*(T_1'^*)$ and $\chi$ is an isomorphism. One concludes that
$  \phi^{-1}(1) = \phi'^{-1}(1) $ and the uniqueness follows.

By the Hochschild-Serre spectral sequence (see Chapter III, Theorem 2.20 in \cite{Milne80}) with $\Pic(X_{\bar k})=\Pic(T_1\times_k \bar{k})=0$, we have
$$ Br_1(X) \cong H^2(k, \bar{k}[X]^\times) \cong H^2(k, \bar{k}[T_1]^\times)\cong Br_1 (T_1)$$
induced by $\phi$.
\end{proof}

The following kind of toric varieties is crucial for studying strong approximation.

\begin{definition} \label{d} A smooth toric variety $(T\hookrightarrow X)$ over $k$ is called of pure divisorial type if the dimension of any $T(\bar{k})$-orbit of $X(\bar{k})$ is $\dim(T)$ or $\dim(T)-1$. Equivalently, the dimension of any cone in the fan of $X$ is strictly less than 2.
\end{definition}

Let us give some examples of smooth toric varieties of pure divisorial type.

\begin{exa} \label{torsor} Any $\Bbb G_m$-torsor $X$ over $\Bbb P^1$ may be given the structure of smooth
toric variety $(\Bbb G_m^2\hookrightarrow X)$ of pure divisorial type.\end{exa}
\begin{proof} Let $\{U_1, U_2\}$ be an open covering of $\Bbb P^1$ such that $$U_1\cong U_2\cong \Bbb A^1 \ \ \ \text{and} \ \ \ U_1\cap U_2\cong \Bbb G_m$$ over $k$ and let $f: X\rightarrow \Bbb P^1$ be a $\Bbb G_m$-torsor. Then $f^{-1}(U_i)$ is a $\Bbb G_m$-torsor over $U_i$, and there are trivializations
\[ \begin{CD} f^{-1}(U_i)   @>{\cong}>> U_i\times_k \Bbb G_m   \\
  @V{f}VV   @VV{p_i}V   \\
 U_i @>>{id}> U_i \end{CD} \]
 where $p_i$ is the projection map for $i=1,2$. We may choose the coordinates
 $$ f^{-1}(U_i)= \Spec (k[t_i, x_i, x_i^{-1}]) $$ for $i=1,2$ such that
 $$ f^{-1}(U_1\cap U_2)= \Spec (k[t_i, t_i^{-1}, x_i, x_i^{-1} ]) $$ and one has the change of coordinates
 \begin{equation} \label{coor} \begin{cases}
  t_1=t_2^{-1}   \\
  x_1=x_2 t_2^n
 \end{cases} \end{equation}
for some $n\in \Bbb Z$. All $\Bbb G_m$-torsors over $\Bbb P^1$ are classified by the integer $n$.

Since $$f^{-1}(U_1\cap U_2)\cong \Bbb G_m \times_k \Bbb G_m$$ is a split torus, one can define an action of $f^{-1}(U_1\cap U_2)$
$$ m_i: \ \ f^{-1}(U_1\cap U_2) \times_k f^{-1}(U_i) \rightarrow f^{-1}(U_i) $$
by sending $t_i\mapsto t_i\otimes t_i$ and $x_i\mapsto x_i\otimes x_i$ for $i=1,2$. This implies that
$(f^{-1}(U_1\cap U_2) \hookrightarrow f^{-1}(U_i))$ is an affine smooth toric variety of pure divisorial type for $1\leq i\leq 2$.  Since $\{f^{-1}(U_i)\}_{i=1,2}$ is an open covering of $X$, one concludes $\{f^{-1}(U_1\cap U_2)\times_k f^{-1}(U_i)\}_{i=1,2}$ is an open covering of $f^{-1}(U_1\cap U_2)\times_k X$. In the common part
$$ [f^{-1}(U_1\cap U_2)\times_k f^{-1}(U_1)]\cap [f^{-1}(U_1\cap U_2)\times_k f^{-1}(U_2)]= f^{-1}(U_1\cap U_2)\times_k f^{-1}(U_1\cap U_2),$$ one has $m_1=m_2$ the multiplication of $f^{-1}(U_1\cap U_2)$. One can glue $m_1$ and $m_2$ along this open set and get an action
$$ m_X: \ \ f^{-1}(U_1\cap U_2)\times_k X \rightarrow X $$ over $k$. This implies that $(f^{-1}(U_1\cap U_2) \hookrightarrow X)$ is a smooth toric variety of pure divisorial type.   \end{proof}

If $n=1$, the corresponding $X$ is a universal $\Bbb G_m$-torsor over $\Bbb P^1$. In this case,  one has $$ f^{-1}(U_1)= \Spec (k[t_1, x_1, x_1^{-1}])= \Spec(k[x_2x_1^{-1}, x_1, x_1^{-1}])=\Spec(k[x_2, x_1, x_1^{-1}]) $$ and
$$ f^{-1}(U_2)= \Spec (k[t_2, x_2, x_2^{-1}])= \Spec(k[x_1x_2^{-1}, x_2, x_2^{-1}])=\Spec(k[x_1, x_2, x_2^{-1}]). $$ This implies that $X\cong \Bbb A^2 \setminus \{(0,0)\}$ over $k$.

\begin{rem}\label{picard} One can further compute $\Pic(X)$ in Example \ref{torsor} by using Proposition 6.10 in \cite{Sansuc}. Indeed, one has the following exact sequence
$$ 1\rightarrow k[X]^\times/k^\times \rightarrow \Bbb G_m^*(k) \rightarrow \Pic(\Bbb P^1) \rightarrow \Pic(X) \rightarrow 1$$ where the map $\Bbb G_m^*(k)\cong \Bbb Z \rightarrow \Pic(\Bbb P^1)$ sends $1$ to $[X]$ (see also p.313 in \cite{CTX}).

If $n=0$ in the equation $(\ref{coor})$, then $k[X]^\times/k^\times \cong \Bbb Z$. This implies that $\Pic(X)\cong \Bbb Z$. In this case, one has $X\cong \Bbb P^1\times_k \Bbb G_m$ over $k$.

Otherwise one has $k[X]^\times =k^\times$ by the equation $(\ref{coor})$. Therefore $\Pic(X)\cong \Bbb Z/n\Bbb Z$, where $n\in \Bbb Z\cong\Pic(X)$ is the element corresponding to $[X]$. This provides a counter-example to Proposition on p.63 in \cite{Fulton} which claims that $\Pic(X)$ is free. Indeed, the corresponding fan $\Delta$ of $X$ in Example \ref{torsor} consists of three cones
$$ \sigma_1=\{ re_1: r\geq 0 \}, \ \ \ \sigma_2=\{ r(-e_1+ne_2): r\geq 0 \} \ \ \ \text{and} \ \ \ \sigma_1\cap \sigma_2=0 $$ where $N=\Bbb Z e_1+\Bbb Z e_2$ is the dual lattice of $T^*$. The condition of Proposition on p.63 in \cite{Fulton} that the fan $\Delta$ is not contained in any proper subspace of $N_\Bbb R$ is equivalent to $n\neq 0$ in this case. Such an example can also be found in \cite{CLS} (p.178) Example 4.2.3 (see also Proposition 4.2.5 in \cite{CLS}).

Lemma 3.5.1 in \cite{ChTs} also claims that a Picard group is torsion free, but this lemma relies on the Proposition at p.63 in \cite{Fulton}.
\end{rem}

\begin{prop}\label{red} If $(T\hookrightarrow X)$ is a smooth toric variety over $k$, there is a unique open toric subvariety $(T\hookrightarrow Y)$ of $(T\hookrightarrow X)$ of pure divisorial type over $k$ such that $\codim(X\setminus Y,X)\geq 2$.
\end{prop}

\begin{proof} Let $m$ be the minimal dimension of all $T(\bar{k})$-orbits in $X(\bar{k})$. One only needs to consider $m < \dim(T)-1$. Since the orbits of the minimal dimension are closed (see Chapter I, \S 1, 1.8 Proposition in \cite{Bo}), the union of all minimal dimensional orbits is closed by Corollary \ref{orb} and $\Gamma_k$-invariant. Therefore there is a closed sub-scheme $W$ of $X$ over $k$ such that $W(\bar{k})$ is the union  of all minimal dimensional orbits by Galois descent. Then $Y_1=X\setminus W$ is an open toric subvariety of $X$ over $k$ and the dimension of any $T(\bar{k})$-orbit of $Y_1(\bar{k})$ is greater than $m$. The existence follows from induction on $Y_1$.

Suppose $Z$ is another open toric subvariety of pure divisorial type of $X$. Since the dimension of $T(\bar k)$ orbits in $Z$ is $\dim(T)$ or $\dim(T)-1$, one has $Z\subseteq Y$ by the above construction. If one further assumes that $\dim(X\setminus Z)<\dim(T)-1$, then $X\setminus Z \subseteq X\setminus Y$ by the above construction. This implies that $Y\subseteq Z$. Therefore $Z=Y$ and the uniqueness follows.
\end{proof}

\begin{lem} \label{prod} If $(T_i\hookrightarrow X_i)$ are smooth toric varieties over $k$ and $(T_i\hookrightarrow Y_i)$ are the unique open toric subvarieties of pure divisorial type with $\codim(X_i\setminus Y_i, X_i)\geq 2$ for $1\leq i\leq n$ respectively, then the unique open toric subvariety $(\prod_{i=1}^n T_i \hookrightarrow Y)$ of pure divisorial type with $$\codim((\prod_{i=1}^n X_i)\setminus Y, \prod_{i=1}^n X_i)\geq 2 \ \ \ \text{in} \ \ \ (\prod_{i=1}^n T_i \hookrightarrow \prod_{i=1}^n X_i ) $$ is given by  $$ Y=\bigcup_{i=1}^n  (T_1\times_k \cdots \times_k T_{i-1} \times_k Y_i \times_k T_{i+1} \times_k\cdots\times_k T_n) .  $$
\end{lem}

\begin{proof} Since
$$\dim((T_1\times_k \cdots \times_k T_n)(\bar{k}) \cdot (x_1, \dots, x_n))=\sum_{i=1}^n \dim(T_i(\bar{k})\cdot x_i) $$ for any $(x_1, \dots, x_n)\in X_1(\bar{k}) \times \cdots \times X_n(\bar{k})$, one obtains that
$$\dim((T_1\times_k \cdots \times_k T_n)(\bar{k}) \cdot (x_1, \dots, x_n))=\dim(T_1\times_k \cdots \times_k T_n)-1 $$ if and only if there is $1\leq i_0\leq n$ such that $$ \dim(T_i(\bar{k}) \cdot x_i)= \begin{cases} \dim(T_i)-1 \ \ \ & \text{if $i=i_0$} \\
\dim(T_i) \ \ \ & \text{otherwise.} \end{cases} $$ This implies that $$\bigcup_{i=1}^n  (T_1\times_k \cdots \times_k T_{i-1} \times_k Y_i \times_k T_{i+1} \times_k\cdots\times_k T_n)$$ is of pure divisorial type and contains all orbits of $\dim(T_1\times_k \cdots \times_k T_n)$ or $\dim(T_1\times_k \cdots \times_k T_n)-1$.
\end{proof}

\begin{definition}\label{standard} Let $d$ be a positive integer, and $k_i/k$ some finite field extensions for $1\leq i\leq d$. We note $K:=\prod_{i=1}^d k_i$. A smooth toric variety $(\Res_{K/k}(\Bbb G_m)\hookrightarrow X)$ over $k$ is called the standard toric vatiety respect to $K/k$, if it is the unique open toric subvariety of pure divisorial type over $k$ in
 $$( \Res_{K/k}(\Bbb G_m)\hookrightarrow \Res_{K/k}(\Bbb A^1)) \ \ \ \text{with} \ \ \  \codim(\Res_{K/k}(\Bbb A^1) \setminus X, \ \Res_{K/k}(\Bbb A^1))\geq 2.$$
\end{definition}

Let $X$ be a smooth toric variety of pure divisorial type with respect to $T$ over $k$ and
\begin{equation} \label{c} X \setminus T=\coprod_{i=1}^d  C_i \ \ \ \text{and} \ \ \ U_i = X \setminus (\coprod_{j\neq i} C_j) \end{equation} for $1\leq i\leq d$, where the $C_i$'s are integral closed sub-schemes of $X$ over $k$ with codimension one. Then $U_i$ is an open toric subvariety of $X$ over $k$ for $1\leq i\leq d$. By Lemma \ref{open}, one obtains that each $T$-orbit in $X$ over $\bar k$ is smooth. Since $C_i$ consists of the $\bar k$-orbits of $T$, one has that $C_i$ is also smooth for $1\leq i\leq d$.

Let $k_i$ be the algebraic closure of $k$ inside $k(C_i)$ for $1\leq i\leq d$. There is a closed geometrically integral sub-scheme $D_i$ over $k_i$ such that \begin{equation} \label{cd}  C_i \times_k \bar k = \coprod_{\sigma\in \Upsilon_i} \sigma (D_i) \end{equation}
where $\Upsilon_i= \Gamma_k/\Gamma_{k_i}$ is the set of all $k$-embedding of $k_i$ into $\bar k$ for $1\leq i\leq d$. Since $\Gamma_{k_i}$ acts on $ \coprod_{\tau \in \Upsilon_i, \tau\neq 1} \tau (D_i)$ stably, one concludes that $\Gamma_{\sigma(k_i)}=\sigma \Gamma_{k_i} \sigma^{-1}$ acts on
$ \coprod_{\tau\in\Upsilon_i, \tau \neq \sigma} \tau (D_i) $ stably for each $\sigma\in \Upsilon_i$. This implies that the scheme $ \coprod_{\tau\in\Upsilon_i, \tau \neq \sigma} \tau (D_i) $ is defined over $\sigma(k_i)$ for each $\sigma\in \Upsilon_i$ by Galois descent.

For each $\sigma\in \Upsilon_i$, one defines
 \begin{equation}\label{Z} \sigma(Z_i) = (X\times_k \sigma(k_i)) \setminus ((\coprod_{\tau\in\Upsilon_i, \tau \neq \sigma} \tau (D_i)) \cup (\coprod_{j\neq i} C_i\times_k \sigma(k_i))) \end{equation} which is an open toric subvariety of $(T\times_k \sigma(k_i)\hookrightarrow X\times_k \sigma(k_i))$ over $\sigma(k_i)$ for $1\leq i\leq d$. Since $D_i$ is geometrically integral, this implies that $\sigma(Z_i)$ contains only two orbits over $\bar k$ for $1\leq i\leq d$. Since $\sigma(Z_i)$ is covered by open affine toric sub-varieties over $\bar k$ by Lemma \ref{open}, the open affine toric sub-varieties which contain the closed orbit must be $\sigma(Z_i)$. This implies that $\sigma(Z_i)$ is affine and $\{ \sigma(Z_i) \times_{\sigma(k_i)} \bar{k} \}_{\sigma\in\Upsilon_i}$ is a smooth open affine toric subvariety covering of $U_i\times_k \bar{k}$ for $1\leq i\leq d$.

By Proposition \ref{str}, the short exact sequence
 \begin{equation} \label{dexact} 1\rightarrow \bar{k}[\sigma(Z_i)]^\times /{\bar k}^\times \xrightarrow{\phi_\sigma^*} \bar{k}[T]^\times /{\bar k}^\times \xrightarrow{\varrho_\sigma^*} \Bbb Z\sigma(D_i) \rightarrow 1 \end{equation}  of $\Gamma_{\sigma(k_i)}$-module given by sending $f$ to its valuation at $\sigma(D_i)$ yields the exact sequence of tori
 \begin{equation} \label{texact} 1\rightarrow \Bbb G_m \xrightarrow{\varrho_\sigma} T\times_k \sigma(k_i) \xrightarrow{\phi_\sigma} T_\sigma \rightarrow 1  \end{equation}
over $\sigma(k_i)$ with $(T_\sigma)^*=\bar{k}[\sigma(Z_i)]^\times /{\bar k}^\times$ and a closed immersion of toric varieties
\begin{equation} \label{h} (\Bbb G_m\hookrightarrow \Bbb A^1) \xrightarrow{\varrho_\sigma} ( T\times_k \sigma(k_i) \hookrightarrow \sigma(Z_i)) \end{equation}  over $\sigma(k_i)$. Moreover the morphism $\phi_\sigma$ can be extended to
\begin{equation} \label{phi} \phi_\sigma: \sigma(Z_i) \rightarrow T_\sigma  \ \ \ \text{with} \ \ \ \varrho_\sigma(\Bbb A^1)=\phi_\sigma^{-1}(1) \end{equation} for any $\sigma\in \Upsilon_i$.

\begin{lem} \label{ext} With the above notation, one considers  the homomorphism of $\Gamma_k$-modules
$$\rho_i^*: \ \  \bar{k}[T]^\times /{\bar k}^\times \rightarrow \Div_{(U_i\times_k \bar k) \setminus T_{\bar k}}(U_i\times_k \bar k) $$
sending $f$ to  $div_{(U_i\times_k \bar k) \setminus T_{\bar k}}(f)$ and obtains a homomorphism $ \Res_{k_i/k}\Bbb G_m  \xrightarrow{\rho_i}  T $ of tori over $k$ for $1\leq i\leq d$. If $(\Res_{k_i/k}\Bbb G_m \hookrightarrow V_i)$ is the standard toric variety respect to $k_i/k$,
  then the homomorphism $\rho_i$ can be extended to a morphism of toric varieties
  $$(\Res_{k_i/k} (\Bbb G_m)\hookrightarrow V_i) \xrightarrow{\rho_i} (T\hookrightarrow U_i) $$ over $k$ for $1\leq i\leq d$.
\end{lem}

\begin{proof} Since $$\rho_i^*(f) =\sum_{\sigma\in \Upsilon_i} \varrho_\sigma^*(f) $$ for any $f\in\bar{k}[T]^\times /{\bar k}^\times$ by (\ref{dexact}) where $\Upsilon_i$ is the set of all $k$-embedding of $k_i$ into $\bar k$, one has
 \begin{equation} \label{mult} \rho_i: \ \Res_{k_i/k}\Bbb G_m(\bar k)=(\bar{k}\otimes_k k_i)^\times= \prod_{\sigma \in \Upsilon_i}\bar{k}^\times \rightarrow T(\bar k) ; \ \ \ (a_{\sigma})_{\sigma\in \Upsilon_i}\mapsto \prod_{\sigma\in \Upsilon_i}  \varrho_\sigma (a_\sigma) \end{equation} for $1\leq i\leq d$. Let
$$ Y_\sigma = Spec(\bar{k}[x_\sigma, x_\tau, x_\tau^{-1}]_{\tau\in \Upsilon_i; \ \tau\neq \sigma}) \subset \Res_{k_i/k} (\Bbb A^1)\times_{k} \bar{k} = Spec(\bar{k}[x_\sigma]_{\sigma\in \Upsilon_i}) $$
over $\bar k$ for each $\sigma\in \Upsilon_i$.  Then $\{Y_\sigma\}_{\sigma\in \Upsilon_i}$ is an open covering of $V_i\times_k \bar k$ for $1\leq i\leq d$.

Applying (\ref{h}) over $\bar k$, one obtains
$$ \varrho_\sigma:  Spec(\bar{k}[x_\sigma])  \rightarrow \sigma(Z_i) \times_{\sigma(k_i)} \bar{k} \subseteq U_i\times_k \bar{k} $$ and $\rho_i$ can be extended to
$$ \rho_i : \ Y_\sigma \rightarrow  \sigma(Z_i) \times_{\sigma(k_i)} \bar{k} \subseteq U_i\times_k \bar{k}$$
for each $\sigma\in \Upsilon_i$. Therefore $\rho_i$ can be extended to $V_i$ for $1\leq i\leq d$.
\end{proof}


Gluing all $\rho_i$ in Lemma \ref{ext} together for $1\leq i\leq d$, one obtains the following proposition.

\begin{prop} \label{rho} Let $(T\hookrightarrow X)$ be a smooth toric variety of pure divisorial type over $k$ and
$$ \rho: \ \  T_0 = \Res_{K/k}(\Bbb G_m) \rightarrow  T $$ be the homomorphism of tori induced by the homomorphism of $\Gamma_k$-modules
$$\rho^*: \ \bar{k}[T]^\times/\bar{k}^\times \rightarrow  \Div_{X_{\bar{k}}\setminus T_{\bar{k}}}(X_{\bar{k}}); \ \ f\mapsto div_{X_{\bar{k}}\setminus T_{\bar{k}}}(f)  $$
where $K=\prod_{i=1}^d k_i$ and $k_i$ is the algebraic closure of $k$ inside $k(C_i)$ with $C_i$ in (\ref{c}).
If $ T_0 = \Res_{K/k}(\Bbb G_m)\hookrightarrow  V$ is the standard toric variety respect to $K/k$,  then $\rho$ can be extended to a morphism of toric varieties $(T_0\hookrightarrow V)\xrightarrow{\rho} (T\hookrightarrow X)$.
\end{prop}

\begin{proof} By Lemma \ref{prod}, one has
$$V= \bigcup_{i=1}^d (\prod_{1\leq j\leq i-1}  \Res_{k_j/k}(\Bbb G_m)\times_k V_i \times_k \prod_{i+1\leq j\leq d}  \Res_{k_j/k}(\Bbb G_m) ) $$ where $V_i$ is given in Lemma \ref{ext} for $1\leq i\leq d$. Define  $$ \aligned g_i : \  & \prod_{1\leq j\leq i-1}  \Res_{k_j/k}(\Bbb G_m)\times_k V_i \times_k \prod_{i+1\leq j\leq d}  \Res_{k_j/k}(\Bbb G_m) \xrightarrow{\rho_1\times \cdots \times \rho_d} T\times_k \cdots \times_k U_i \times_k \cdots \times_k T  \\
& \xrightarrow{id\times \cdots \times i_{U_i}\times \cdots \times id} T\times_k \cdots \times_k X \times_k \cdots \times_k T \xrightarrow{m_X} X  \endaligned $$
where $i_{U_i}$ is the open inclusion $U_i\subseteq X$ and $\rho_i$ is given in Lemma \ref{ext} and $m_X$ is the action of $T$ for $1\leq i\leq d$. Since
$ \rho^* = \oplus_{i=1}^d \rho_i^* $, one concludes that $g_i|_{T_0}=\rho$ for $1\leq i\leq d$. Therefore the morphisms $\{ g_i\}_{1\leq i\leq d}$ can be glued together to obtain the required morphism. \end{proof}

By purity (see the end of p.24 in \cite{CT93}) and Lemma \ref{red}, one only needs to compute the Brauer groups of smooth toric varieties of pure divisorial type.

\begin{prop} \label{br} One has the following exact sequence
$$0\rightarrow  \Br_a(X)\rightarrow \Br_a(T)\xrightarrow{\rho^*} \Br_a (T_0) $$ for a smooth toric variety $(T\hookrightarrow X)$ of pure divisorial type over $k$, where $\rho$ and $T_0$ are given by Proposition \ref{rho} and $\rho^*$ is the induced by $\rho$.
\end{prop}


\begin{proof}
From Colliot-Th\'el\`ene and Sansuc \cite{CTS87} \S 1 (see also Diagram 4.15 in \cite{Sko}), we have a commutative diagram with exact rows and exact columns
$$\xymatrix{0\ar[r]&H^2(k,\bar{k}[X]^\times/\bar{k}^\times)\ar[r]\ar[d]&Br_a(X)\ar[r]\ar[d]&H^1(k,Pic(X_{\bar{k}}))\ar[d]\\
0\ar[r]&H^2(k,T^*)\ar[r]^{h_1}\ar[d]^{h_2}&Br_a(T)\ar[r]\ar[d]^{h_3}&H^1(k,Pic(T_{\bar{k}}))=0\ar[d]\\
0\ar[r]&H^2(k, \bar{k}[T]^\times/\bar{k}[X]^\times)\ar[r]^{h_4}&H^2(k,Div_{X_{\bar{k}}-T_{\bar{k}}}(X_{\bar{k}}))\ar[r] & H^2(k,Pic_{X_{\bar{k}}-T_{\bar{k}}}(X_{\bar{k}})) }.$$
Since $T_0^*\cong Div_{X_{\bar{k}}-T_{\bar{k}}}(X_{\bar{k}})$, the result follows from that fact that $h_3\circ h_1=h_4\circ h_2$ is induced by $\rho^*: T^*\rightarrow T_0^*$.
\end{proof}

\section{Relative strong approximation for tori}

In this section, we extend strong approximation with Brauer-Manin obstruction off $\infty_k$ for tori proved by Harari in \cite{Ha08} to the relative situation. In \cite{D}, Demarche used a similar idea for studying hyper-cohomology of complexes of two tori with finite kernel to establish strong approximation with Brauer-Manin obstruction off $\infty_k$ for reductive groups.

\begin{definition} Let $X$ be a separated integral scheme of finite type over $k$. An integral model $\bf X$ of $X$ over $O_k$ (or $O_{k,S}$ for some finite subset S of $\Omega_k$ containing $\infty_k$) is defined to be a separated integral scheme of finite type over $O_k$ (or $O_{k,S}$) such that ${\bf X}\times_{O_k} k=X$ (or ${\bf X}\times_{O_{k,S}} k=X$).

If $T$ is a group of multiplicative type over $k$, an integral model $\bf T$ of $T$ over $O_{k}$ (or $O_{k,S}$) is defined to be an integral model of $T$ which is a group scheme of multiplicative type over $O_k$ (or $O_{k,S}$) extended from $T$.
\end{definition}

Let $X$ be a separated integral scheme of finite type over $k$ and $\pi_0(X(k_v))$ be the set of connected components of $X(k_v)$ for each $v\in \infty_k$. Define
$$ X(\mathbf A_k)_{\bullet}= [\prod_{v\in \infty_k} \pi_0(X(k_v))]\times X(\mathbf A_k^\infty) $$ and
$$ X({\mathbf A}_k)_{\bullet}^B = \{ (x_v)_{v\in \Omega_k}\in X({\mathbf A}_k)_{\bullet}: \ \ \sum_{v\in \Omega_k} \inv_v(\xi(x_v))=0, \ \ \forall \xi\in B \}  $$
for any finite subset $B$ of $Br_a(X)$. This is well-defined because any element in $Br_a(X)$ takes a constant value on each connected component of $X(k_v)$ for any $v\in \infty_k$.

\begin{lem}\label{image} Let $\psi: T_1 \rightarrow T_2$ be a homomorphism of tori. Then $\psi(T_1(k_v))$ is closed in $T_2(k_v)$ for all $v\in \Omega_k$.
\end{lem}
\begin{proof} Let $T$ be the image of $\psi$. For any $v\in \Omega_k$, one has that $\psi(T_1(k_v))$ is an open subgroup of $T(k_v)$ by corollary 1 of Chapter 3 in \cite{PR}. Therefore $\psi(T_1(k_v))$ is closed in $T(k_v)$. It is clear that $T(k_v)$ is closed in $T_2(k_v)$. One concludes that $\psi(T_1(k_v))$ is closed in $T_2(k_v)$.
\end{proof}

\begin{prop}\label{top} With the same notation as that in Lemma \ref{image}, one has
$$\psi(T_1(\mathbf A_k))= (\prod_{v\in \Omega_k} \psi(T_1(k_v))) \cap T_2(\mathbf A_k) \subseteq \prod_{v\in \Omega_k} T_2(k_v). $$
 In particular, $\psi(T_1(\mathbf A_k))$ is closed in $T_2(\mathbf A_k)$.
\end{prop}

\begin{proof}  If $\psi$ is surjective, one has the short exact sequence of groups of multiplicative type
$$ 1\rightarrow T_0 \rightarrow T_1 \xrightarrow{\psi} T_2 \rightarrow 1 $$ with $T_0=ker \psi$. There is a finite subset $S$ of $\Omega_k$ containing $\infty_k$ such that the above short exact sequence extends to
$$  1\rightarrow {\bf T}_0 \rightarrow  {\bf T}_1 \xrightarrow{\psi_S} {\bf T}_2 \rightarrow 1$$  over $O_{k,S}$, where ${\bf T}_0$, ${\bf T}_1$ and ${\bf T}_2$ are integral models of $T_0$, $T_1$ and $T_2$ over $O_{k,S}$ respectively. For $v\not\in S$, this yields exact sequences:
\[ \begin{CD} 1 @>>>  {\bf T}_0(O_v) @>>> {\bf T}_1(O_v) @>{\psi_S}>> {\bf T}_2(O_v) @>>> H^1(O_v, {\bf T}_0)  \\
                @. @VVV   @VVV   @VVV  @VVV \\
       1 @>>> T_0(k_v) @>>> T_1(k_v) @>>{\psi}> T_2(k_v) @>>> H^1(k_v, T_0)
       \end{CD} .\]
 By Proposition 2.2 in \cite{CTS}, the natural map $H^1(O_v, {\bf T}_0) \rightarrow  H^1(k_v, T_0)$ is injective.
Then
 $$\psi(T_1(k_v))\cap {\bf T}_2(O_v) = \psi_S ({\bf T}_1(O_v)) $$ for all $v \not\in S$. Therefore
$$\psi(T_1(\mathbf A_k))= (\prod_{v\in \Omega_k} \psi(T_1(k_v))) \cap T_2(\mathbf A_k) . $$

In general, there is a closed sub-torus $T$ of $T_2$ such that $\psi$ factors through the surjective homomorphism $T_1 \rightarrow T$. By the above arguments, one has $$\psi(T_1(\mathbf A_k))= (\prod_{v\in \Omega_k} \psi(T_1(k_v)))\cap T(\mathbf A_k) $$ and $\psi(T_1(\mathbf A_k))$ is closed in $T(\mathbf A_k)$. Since $T$ is a closed sub-torus of $T_2$, one has $$T(\mathbf A_k)= (\prod_{v\in \Omega_k} T(k_v)) \cap T_2(\mathbf A_k) $$ and $T(\mathbf A_k)$ is a closed subset of $T_2(\mathbf A_k)$. Therefore one concludes that $$ \psi(T_1(\mathbf A_k))= (\prod_{v\in \Omega_k} \psi(T_1(k_v)))\cap T_2(\mathbf A_k) $$ and $\psi(T_1(\mathbf A_k))$ is closed in $T_2(\mathbf A_k)$ by Lemma \ref{image}. \end{proof}

By the functoriality of \'etale cohomology, one obtains an induced group homomorphism $$ \psi_{\Br}^*: \ \ \Br_a (T_2) \longrightarrow \Br_a(T_1) $$ for any homomorphism $\psi: T_1\rightarrow T_2$ of tori. For each $v\in \infty_k$, since the map $\psi$ maps each connected component of $T_1(k_v)$ into one connected component of $T_2(k_v)$, one has
$$\psi(T_1(\mathbf A_k)_{\bullet}) \subseteq T_2(\mathbf A_k)_{\bullet}^{ker(\psi_{\Br}^*)}  $$  by the functoriality of Brauer-Manin pairing (see Page 102, (5.3) in \cite{Sko}). One can extend strong approximation for tori proved by Harari in \cite{Ha08} to the following relative strong approximation for tori.

\begin{prop}\label{relative} Let $\psi: T_1\rightarrow T_2$ be a homomorphism of tori with $\Sha^1(T_1) =0 $. Then the image of $T_2(k)$ is dense in
$$ T_2(\mathbf A_k)_\bullet ^{ker(\psi_{\Br}^*)}/\psi(T_1(\mathbf A_k)_\bullet) $$ with the quotient topology.
\end{prop}

\begin{proof}
By Theorem 2 in \cite{Ha08} and functoriality, one has the following commutative diagram of exact sequences
\[ \begin{CD} 0 @>>> \overline{T_1(k)} @>>> T_1(\mathbf A_k)_\bullet @>>> \Br_a(T_1)^D @>>> \Sha^1(T_1) =0  \\
 @. @VVV   @VVV @VVV @. \\
0 @>>> \overline{ T_2(k) } @>>> T_2(\mathbf A_k)_\bullet @>>> \Br_a(T_2)^D @.
\end{CD} \]
where $\overline{T_1(k)}$ and $\overline{ T_2(k)}$ are the topological closure of $T_1(k)$ and $T_2(k)$ in $T_1(\mathbf A_k)_\bullet$ and $T_2(\mathbf A_k)_\bullet$ respectively and  $$\Br_a(T_i)^D= \Hom (\Br_a(T_i), \Bbb Q/\Bbb Z)$$ for $i=1, 2$. Since $\Bbb Q/\Bbb Z$ is an injective $\Bbb Z$-module, one has $\Hom (*, \Bbb Q/\Bbb Z)$ is an exact functor and the sequence
$$ \Br_a(T_1)^D \rightarrow \Br_a(T_2)^D \rightarrow ker(\psi_{\Br}^*)^D\rightarrow 0 $$ is exact. Therefore the natural map
$$\overline{T_2(k)} \rightarrow T_2(\mathbf A_k)_\bullet ^{ker(\psi_{\Br}^*)}/\psi(T_1(\mathbf A_k)_\bullet)$$ is surjective by the snake lemma. Since the topological closure of the image of $T_2(k)$ in
$$T_2(\mathbf A_k)_\bullet ^{ker(\psi_{\Br}^*)}/\psi(T_1(\mathbf A_k)_\bullet)$$ with the quotient topology contains the image of $\overline{T_2(k)}$ by Proposition \ref{top}, one obtains the result as desired.
\end{proof}

\begin{rem}\label{relarem} One can state Proposition \ref{relative} in the following equivalent version for better understanding of relative strong approximation.

If $$ [(\prod_{v\in \infty_k} a_v N_{\Bbb C/k_v}(T_2(\Bbb C))) \times U]\cap T_2(\mathbf A_k)^{ker(\psi_{\Br}^*)} \neq \emptyset , $$
for an open subset $U$ of $T_2(\mathbf A_k^\infty)$ and $a_v\in T(k_v)$ with $v\in \infty_k$, then there are $x\in T_2(k)$ and $y\in T_1(\mathbf A_k)$ such that $$x\psi(y)\in (\prod_{v\in \infty_k} a_v N_{\Bbb C/k_v}(T_2(\Bbb C))) \times U . $$
\end{rem}

\bigskip

In order to prove our main result, we need the following useful result.

\begin{prop}\label{cdim}
Let $S$ be a finite nonempty subset of $\Omega_k$, and $U$ an open subscheme of $\Bbb A^n$ with $\codim(\Bbb A^n\setminus U, \Bbb A^n)\geq 2$. Then $U$ satisfies strong approximation off $S$.
\end{prop}

\begin{proof} Since the projection $$p: \ \Bbb A^n \rightarrow \Bbb A^1;  \ \ \ (x_1,...x_n)\mapsto x_1  $$ with
 $ p^{-1}(x)\cong \Bbb A^{n-1}$ over $k$, one has
  $$  \sharp  \{ x\in \Bbb A^1(k):   \dim( p^{-1}(x)\cap Z)= \dim(Z) \} < \infty $$ with $Z=\Bbb A^n \setminus U$.
  Thus for almost all $x\in \Bbb A^1(k)$, $\codim(p^{-1}(x)\cap Z, p^{-1}(x))\geq 2$, and one obtains that $p^{-1}(x)\cap U$ satisfies strong approximation off $S$ by induction.

  For any $x\in \Bbb A^1(\bar k)$,
  $$p^{-1}(x)\cap U= p^{-1}(x) \setminus (p^{-1}(x)\cap Z) \neq \emptyset  $$
  and geometrically integral since $\dim(p^{-1}(x))> \dim(p^{-1}(x)\cap Z)$. Since $p^{-1}(x)(k_v)$ is Zariski dense in $p^{-1}(x)$ for any $x\in \Bbb A^1(k_v)$ (see Theorem 2.2 in Chapter 2 of \cite{PR}), one has
  $$ (p^{-1}(x)\cap U)(k_v) = p^{-1}(x)(k_v) \setminus (p^{-1}(x)\cap Z)(k_v) \neq \emptyset $$
  for any $v$. This implies that condition (iii) of Proposition 3.1 in \cite{CTX1} is satisfied. The result follows from Proposition 3.1 in \cite{CTX1}.
\end{proof}

\begin{cor} \label{v} Let $d$ be a positive integer, $S$ a finite nonempty subset of $\Omega_k$, and $k_i/k$ some finite field extensions for $1\leq i\leq d$. We note $K:=\prod_{i=1}^d k_i$. Then the standard toric variety $(\Res_{K/k}(\Bbb G_m)\hookrightarrow X)$ satisfies strong approximation off $S$.
\end{cor}

\begin{proof} There exists an isomorphism $\Res_{K/k}(\Bbb A^1)\stackrel{\sim}{\rightarrow}\Bbb A^{\sum_{i=1}^d [k_i:k]}$. The result holds from Proposition \ref{cdim}. \end{proof}

 \section{Proof of main theorem}

In this section, we keep the same notation as in the previous sections. Let $(T\hookrightarrow X)$ be a smooth toric variety of pure divisorial type over $k$.

$\bullet$ Fix integral models $\bf X$, $\bf T$, ${\bf C}_i$, ${\bf U}_i$ of $X$, $T$, $C_i$, $U_i$ in (\ref{c}) and ${\bf V}_i$ of $V_i$ in Lemma \ref{ext} and $\bf V$ of $V$ in Proposition \ref{rho} over $O_k$ for $1\leq i\leq d$ respectively.

$\bullet$ Fix integral models $\sigma({\bf Z}_i)$ of $\sigma(Z_i)$ in (\ref{Z}) and ${\bf T}_\sigma$ of $T_\sigma$ in (\ref{phi}) over $O_{\sigma(k_i)}$ for $1\leq i\leq d$ and $\sigma\in \Upsilon_i$, where $\Upsilon_i$ is the set of all $k$-embedding of $k_i$ into $\bar k$.

\medskip

Choose a sufficiently large finite subset $S\supset \infty_k$ in $\Omega_k$  such that

i) The action $m_X$ of $T$ on $X$ as toric variety extends to $$m_T: ({\bf T}\times_{O_k} O_{k,S}) \times_{O_{k,S}} ({\bf X}\times_{O_k} O_{k,S}) \rightarrow {\bf X}\times_{O_k} O_{k,S}  $$ as an action of group scheme.

ii)For $1\leq i\leq d$, $\{ {\bf U}_i\times_{O_k} O_{k,S} \}_{i=1}^d $ is an open covering of ${\bf X}\times_{O_k} O_{k,S}$ and ${\bf U}_i\times_{O_k} O_{k,S}$ is covered by
$$\xymatrix{{\bf T}\times_{O_k} O_{k,S}  \ar[r]^j  &{\bf U}_i \times_{O_k} O_{k,S}& {\bf C}_i \times_{O_k} O_{k,S}\ar[l]^l}  $$
over $O_{k,S}$, where $j$ is an open immersion and $l$ is the complement of $j$, which is a closed immersion. Moreover, ${\bf C}_i\times_{O_k} O_{k,S}$ is smooth over $O_{k,S}$ for $1\leq i\leq d$.  \medskip

Let $O_{k_i,S}$ and $O_{\sigma(k_i),S}$ be the integral closures of $O_{k,S}$ inside $k_i$ and $\sigma(k_i)$ respectively for $\sigma \in \Upsilon_i$ and $1\leq i\leq d$. \medskip

iii) The morphism $\rho$ in Proposition \ref{rho} extends to $\rho: {\bf V}\times_{O_k} O_{k,S} \rightarrow {\bf X}\times_{O_k} O_{k,S}$ and
$$\{\prod_{1\leq j\leq i-1}  \Res_{O_{k_j,S}/O_{k,S}}(\Bbb G_m)\times_{O_{k,S}} ({\bf V}_i\times_{O_k} O_{k,S}) \times_{O_{k,S}} \prod_{i+1\leq j\leq d}  \Res_{O_{k_j,S}/O_{k,S}}(\Bbb G_m)\}_{1\leq i\leq d} $$ is an open covering of ${\bf V}\times_{O_k} O_{k,S}$. \medskip

iv) Both morphism $\varrho_\sigma$ in (\ref{h}) and morphism $\phi_\sigma$ in (\ref{phi}) extend to
$$ \varrho_\sigma: \ \Bbb A_{O_{\sigma(k_i),S}}^1 \rightarrow \sigma({\bf Z}_i) \times_{O_{\sigma(k_i)}} O_{\sigma(k_i), S} \ \ \ \text{and} \ \ \ \phi_\sigma: \ \sigma({\bf Z}_i) \times_{O_{\sigma(k_i)}} O_{\sigma(k_i), S} \rightarrow {\bf T}_\sigma \times_{O_{\sigma(k_i)}} O_{\sigma(k_i), S} $$
over $O_{\sigma(k_i),S}$ for all $\sigma\in \Upsilon_i$ and $1\leq i\leq d$. Moreover, the exact sequence in (\ref{texact}) extends to
$$ 1\rightarrow \Bbb G_{m,O_{\sigma(k_i),S}}  \rightarrow {\bf T}\times_{O_k} O_{\sigma(k_i),S} \rightarrow {\bf T}_\sigma \times_{O_{\sigma(k_i)}} O_{\sigma(k_i), S}  \rightarrow 1$$ over $O_{\sigma(k_i), S}$ and $Im(\varrho_\sigma)=\phi_\sigma^{-1}( {\bf 1}_{{\bf T}_\sigma \times_{O_{\sigma(k_i)}} O_{\sigma(k_i), S} }) )$ over $O_{\sigma(k_i), S}$ for $1\leq i\leq d$ and all $\sigma\in \Upsilon_i$. \medskip

Let $O_{\bar{k},S}$ be the integral closure of $O_{k,S}$ inside $\bar k$. \medskip

v) $${\bf C}_i \times_{O_k} O_{\bar{k},S} = \coprod_{\sigma \in \Upsilon_i} ((\sigma({\bf Z}_i)\setminus {\bf T}) \times_{O_{\sigma(k_i)}} O_{\bar{k}, S}) $$ and $(\sigma({\bf Z}_i)\setminus {\bf T}) \times_{O_{\sigma(k_i)}} O_{\bar{k}, S}$ is integral for $1\leq i\leq d$.   \medskip

vi) The morphism $\rho_i$ in Lemma \ref{ext} extends to the following commutative diagram
\[ \begin{CD}      \Res_{O_{k_i,S}/O_{k,S}}(\Bbb G_m)  @>{\rho_i}>> {\bf T}\times_{O_k} O_{k,S}   \\
 @VVV       @VVV \\
  {\bf V}_i\times_{O_k} O_{k,S} @>>{\rho_i}> {\bf U}_i \times_{O_k} O_{k,S}
\end{CD} \]
over $O_{k,S}$ and $\{  Spec(O_{\bar{k},S}[x_\sigma, x_\tau, x_\tau^{-1}]_{\tau\in \Upsilon_i; \ \tau\neq \sigma})\}_{\sigma\in \Upsilon_i}$ is an open covering of ${\bf V}_i\times_{O_k} O_{\bar{k},S}$ for $1\leq i\leq d$.

The following proposition is crucial for proving our main theorem.

\begin{prop}\label{int} With notation as above, one has
$${\bf X}(O_v)\cap T(k_v) ={\bf T}(O_v)\cdot \rho ({\bf V}(O_v)\cap T_0(k_v))\subseteq X(k_v)$$ for all $v \not\in S$, where $T_0=  \prod_{i=1}^d \Res_{k_i/k}(\Bbb G_m)$.
\end{prop}

\begin{proof} By the above conditions i) and iii), one only needs to prove $${\bf X}(O_v) \cap T(k_v)\subseteq {\bf T}(O_v)\cdot \rho ({\bf V}(O_v)\cap T_0(k_v)). $$ Let $T_i= \Res_{k_i/k}(\Bbb G_m)$ for $1\leq i\leq d$. Since $$ \rho_i({\bf V}_i(O_v)\cap T_i(k_v)) \subseteq \rho ({\bf V}(O_v)\cap T_0(k_v))$$ by the above condition iii) and vi), it is sufficient to show that
 $${\bf U}_i(O_v)\cap T(k_v) \subseteq {\bf T}(O_v)\cdot  \rho_i({\bf V}_i(O_v)\cap T_i(k_v))$$ for each $1\leq i\leq d$ by the above condition ii).

Let $\alpha \in ({\bf U}_i(O_v)\cap T(k_v)) \setminus {\bf T}(O_v)$. Then the special point of $\alpha$ is contained in ${\bf C}_i\times_{O_k} O_v$ by the above condition ii). Then ${\bf C}_i\times_{O_k} O_v$ contains an $O_v$-point $\beta$ with the same specialization as $\alpha$ by the smoothness of ${\bf C}_i\times_{O_k} O_v$.

Fix a prime $w$ in $\bar k$ above $v$. Extending the condition v) to the ring of integers $O_{\bar{k}_w}$ of $\bar k_w$, one obtains $\sigma_\alpha\in \Upsilon_i$ such that $(\sigma_\alpha({\bf Z}_i) \setminus {\bf T})\times_{O_{\sigma_\alpha (k_i)}} O_{\bar{k}_w}$ is the unique connected component containing $\beta$. This implies that $Gal(\bar k_w/k_v)$ acts on $(\sigma_\alpha({\bf Z}_i) \setminus {\bf T})\times_{O_{\sigma_\alpha(k_i)}} O_{\bar{k}_w}$ stably. Therefore $\sigma_\alpha({\bf Z}_i)$ is defined over $O_v$ by Galois descent and $\sigma_{\alpha}(D_i)$ is defined over $k_v$. Since $Gal(\bar{k}/k)$ acts on $\{\sigma (D_i)\}_{\sigma \in \Upsilon_i}$ transitively and the stabilizer of $\sigma_\alpha (D_i)$ is $Gal(\bar{k}/\sigma_\alpha(k_i))$. On the other hand, the closed subgroup $Gal(\bar{k}_w/k_v)$ acts trivially on $\sigma_\alpha (D_i)$. One concludes that $\sigma_\alpha (k_i)\subseteq k_v$ and $O_{\sigma(k_i),S}\subset O_v$. Therefore all morphisms in the condition iv) can be extended to $O_v$.

Since $H_{et}^1(O_v, \Bbb G_m)=0$, one concludes that the homomorphism $$\phi_{\sigma_\alpha}: \  {\bf T}(O_v) \rightarrow {\bf T}_{\sigma_\alpha}(O_v) $$ is surjective by the above condition iv). There is $t\in {\bf T}(O_v)$ such that $$t \cdot \alpha\in \phi_{\sigma_\alpha}^{-1}(1)=Im(\varrho_{\sigma_\alpha})$$ over $O_v$ by the above condition iv). This implies that there is $\gamma \in \Bbb A^1_{O_v}(O_v)=O_v$ such that $\varrho_{\sigma_\alpha}(\gamma)=t \cdot \alpha$. Since $\alpha \in T(k_v)$, one has that $\gamma\neq 0$. Define $$\delta=(\delta_{\sigma})_{\sigma\in \Upsilon_i}\in {\bf V}_i(O_{\bar k_v})\subseteq \prod_{\sigma\in \Upsilon_i} O_{\bar k_v}  $$ as follows
 $$ \delta_\sigma= \begin{cases} \gamma \ \ \ & \text{if $\sigma=\sigma_\alpha$} \\
1 \ \ \ & \text{otherwise}
\end{cases} $$
Since $Gal(\bar{k}_v/k_v)$ acts on $\Upsilon_i$ but fixes $\sigma_\alpha$, one has $\delta\in {\bf V}_i(O_v)\cap T_i(k_v)$ by the above condition vi) and Galois descent. Therefore $$\rho_i(\delta)=\varrho_{\sigma_\alpha}(\gamma)=\alpha\cdot t$$ as desired by the formula (\ref{mult}). \end{proof}

The following local approximation enables us to consider ${\bf X}(O_v)\cap T(k_v)$ instead of ${\bf X}(O_v)$.

\begin{prop} \label{inv} Let $(T\hookrightarrow X)$ be a smooth toric variety over $k_v$ with $v\in \Omega_k$. If $x\in X(k_v)\setminus T(k_v)$, then there is $y\in T(k_v)$ such that $y$ is as close to $x$ as required and $$\inv_v(\xi(x))=\inv_v(\xi(y))$$ for all $\xi\in \Br_1(X)$. \end{prop}
\begin{proof} By corollary \ref{affine}, there is an open affine smooth toric subvariety $M$ of $X$ such that $x\in M(k_v)$. By Proposition \ref{str}, there are finite extensions $E_i/k_v$ such that
 $$ (\prod_i \Res_{E_i/k_v}(\Bbb G_m) \hookrightarrow \prod_i \Res_{E_i/k_v}(\Bbb A^1))$$
is a closed toric subvariety of $(T\hookrightarrow M)$ and the quotient homomorphism $$\phi: \ T\rightarrow T_1 \ \ \ \text{with} \ \ \ T_1=T/(\prod_i \Res_{E_i/k_v}(\Bbb G_m) )$$ can be extended to $\phi: M\rightarrow T_1$ and $\phi^{-1}(1)=\prod_i \Res_{E_i/k_v}(\Bbb A^1)$.

By Shapiro's Lemma and Hilbert 90, one has the map $T(k_v)\xrightarrow{\phi} T_1(k_v)$ is surjective. There is $\alpha\in T(k_v)$ such that $\phi(x)=\phi(\alpha)$. This implies that  $$\alpha^{-1} x \in (\phi^{-1}(1))(k_v)=\prod_i E_i. $$ Choose $z'\in \prod_i E_i^\times$ close to $\alpha^{-1} x$ such that $y=\alpha \cdot z'$ is as close to $x$ as required.

For any $\xi\in \Br_1(X)$, there are $\eta\in \Br_1(T_1)$ such that
$$\phi^*(\eta)=\xi$$
by $ \Br_1(X)\hookrightarrow  \Br_1(M) \stackrel{\sim}{\leftarrow}\Br_1(T_1)$ and Proposition \ref{str}. Since $\phi(x)=\phi(y)=\phi(\alpha)$, one has
$$\inv_v(\eta(\phi(x)))=\inv_v(\eta(\phi(y))) . $$
 By functoriality, this implies
 $$ \inv_v(\phi^*(\eta)(x))=\inv_v (\phi^*(\eta)(y)).$$
 Since
 $$\inv_v(\phi^*(\eta)(x))=\inv_v(\xi(x)) \ \ \ \text{and} \ \ \ \inv_v (\phi^*(\eta)(y))= \inv_v(\xi(y)), $$
  one obtains the result as desired. \end{proof}

\begin{prop}\label{typed} If $X$ is a smooth toric variety of pure divisorial type, then $X$ satisfies strong approximation with Brauer-Manin obstruction off $\infty_k$. \end{prop}

\begin{proof} For any non-empty open subset $\Xi \subseteq X(\mathbf A_k)^{\Br_1 X}$, there are a sufficiently large finite subset $S_1$ of $\Omega_k$ containing $S$ and an open subset $W=\prod_{v\in \Omega_k} W_v$ of $X(\mathbf A_k)$ such that
 $$ \emptyset \neq  W\cap X(\mathbf A_k)^{\Br_1 X} \subseteq \Xi ,$$
and $W_v={\bf X}(O_v)$ for all $v\not\in S_1.$

Let $ (x_v)_{v\in \Omega_k} \in  W\cap X(\mathbf A_k)^{\Br_1 X}$. By Proposition \ref{inv}, one can assume that $x_v\in T(k_v)$ for all $v\in \Omega_k$. Then  $$x_v\in W_v \cap T(k_v)={\bf X}(O_v)\cap T(k_v) ={\bf T}(O_v)\cdot \rho ({\bf V}(O_v)\cap T_0(k_v))$$ for $v\not\in S_1$ by Proposition \ref{int}, where $T_0=  \prod_{i=1}^d \Res_{k_i/k}(\Bbb G_m)$. Let $$t_v\in {\bf T}(O_v) \ \ \ \text{and} \ \ \ \beta_v\in {\bf V}(O_v)\cap T_0(k_v)$$ such that $x_v=t_v \cdot \rho(\beta_v)$ for all $v\not\in S_1$ and $t_v=x_v$ for $v\in S_1$. Then $(t_v)_{v\in \Omega_k}\in T(\mathbf A_k)$.

Since $t_v$ induces a morphism $X\times_k k_v\rightarrow X \times_k k_v$ for all $v\in \Omega_k$, one has
$$ \inv_v(\xi(x_v))= \inv_v(\xi(t_v\cdot \rho(\beta_v)))=\inv_v((\rho^* t^* \xi)(\beta_v))$$ and $$ \inv_v(\xi(t_v))= \inv_v((\rho^* t^* \xi)(1_{T_0}))  $$ for all $\xi\in \Br_1(X)$. By the purity of Brauer groups (see Theorem 6.1 of  Part III in \cite{G}), one has $Br_1(V\times_k k_v)=Br(k_v)$. Therefore $ \inv_v(\xi(x_v))=\inv_v(\xi(t_v))$ for all $v\in \Omega_k$.

By Proposition \ref{br} and Proposition \ref{relative} or Remark \ref{relarem}, there are $t\in T(k)$ and $y_\mathbf A\in T_0(\mathbf A_k)$ such that $$t\rho (y_\mathbf A)\in (\prod_{v\in \infty_k} T(k_v)\times \prod_{v\in S_1\setminus \infty_k} (W_v\cap T(k_v)) \times \prod_{v\not\in S_1} {\bf T}(O_v)) . $$ Therefore the open subset of $V(\mathbf A_k)$
$$ \rho^{-1}(t^{-1} (\prod_{v\in \infty_k} X(k_v) \times \prod_{v\not\in \infty_k} W_v)) $$ contains $y_\mathbf A$ and is not empty. Then there is $$y\in V(k) \cap \rho^{-1}(t^{-1} (\prod_{v\in \infty_k} X(k_v) \times \prod_{v\not\in \infty_k} W_v)) $$ by Corollary \ref{v}. This implies that
$$t \cdot \rho(y) \in (\prod_{v\in \infty_k} X(k_v) \times \prod_{v\not\in \infty_k} W_v) $$ as desired. \end{proof}

For general smooth toric varieties, one needs to extend a part of Proposition \ref{str} to integral models.

\begin{lem}\label{integ} Suppose an affine smooth toric variety $(T\hookrightarrow X)$ over $k_v$ can be extended to an open immersion $\bf T\hookrightarrow \bf X$ over $O_v$ such that $\bf T$ is a torus over $O_v$ and $\bf X$ is an affine scheme of finite type over $O_v$ for $v<\infty_k$. If the base change of the above open immersion fits into a commutative diagram
\[ \begin{CD}   {\bf T}\times_{O_v} {O_v}^{ur}  @>>> {\bf X}\times_{O_v} O_{v}^{ur}   \\
 @V{\cong}VV       @VV{\cong}V \\
 \Bbb G_{m, O_v^{ur}}^{s+t} @>>>  \Bbb A_{O_v^{ur}}^s \times_{O_v^{ur}} \Bbb G_{m, O_v^{ur}}^t
\end{CD} \]
over $O_v^{ur}$, where $O_v^{ur}$ is the ring of integers of the maximal unramified extension $k_v^{ur}$ of $k_v$ such that the left vertical arrow is an isomorphism of group schemes over $O_v^{ur}$, then one has the following commutative diagram

\[ \begin{CD}   \prod_{i=1}^h   \Res_{O_{k_i}/O_{v}}(\Bbb G_{m, O_{k_i}})  @>{\iota}>> {\bf T}   \\
 @VVV       @VVV \\
 \prod_{i=1}^h \Res_{O_{k_i}/O_v} (\Bbb A_{O_{k_i}}^1) @>>>  {\bf X}
\end{CD} \]
where the horizontal arrows are closed immersions and the vertical arrows are open immersions and $O_{k_i}$'s are the rings of integers of finite unramified extensions $k_i/k_v$ for $1\leq i\leq h$. Moreover $\iota$ is a homomorphism of commutative group schemes over $O_v$ and the quotient map $\phi: {\bf T}\rightarrow coker(\iota)$ can be extended to $\phi: {\bf X} \rightarrow coker(\iota)$ such that $$\phi^{-1}(1) = \prod_{i=1}^h \Res_{O_{k_i}/O_v} (\Bbb A_{O_{k_i}}^1) $$ over $O_v$.
\end{lem}

\begin{proof}

Since $\Pic({\bf X}\times_{O_v} O_v^{ur})=0$, one has the following short exact sequence
$$ 1\rightarrow O_v^{ur}[{\bf X}]^\times /{O_v^{ur}}^\times \xrightarrow{\phi^*}  O_v^{ur}[{\bf T}]^\times /{O_v^{ur}}^\times \xrightarrow{\iota^*} \Div_{({\bf X}\times_{O_v} O_v^{ur}) \setminus ({\bf T} \times_{O_v} O_v^{ur})}({\bf X}\times_{O_v} O_v^{ur}) \rightarrow 1 $$ of $\Gal(k_v^{ur}/k_v)$-module by sending $f\mapsto div_{({\bf X}\times_{O_v} O_v^{ur}) \setminus ({\bf T} \times_{O_v} O_v^{ur})}(f)$ for any $f\in  O_v^{ur}[{\bf T}]^\times$. By Theorem 1.2 and  Theorem 3.1 in Expos\'e VIII of \cite{SGA3}, one obtains an exact sequence of affine group schemes
$$1\rightarrow  \prod_{i=1}^h \Res_{O_{k_i}/O_{v}}(\Bbb G_{m, O_{k_i}}) \xrightarrow{\iota} {\bf T}\xrightarrow{\phi} coker(\iota)\rightarrow 1  $$ over $O_v$ where $O_{k_i}$'s are the rings of integers of the finite unramified extensions $k_i/k_v$ for $1\leq i\leq h$, and where $coker(\iota)$ is a torus over $O_v$ with
$$ \Hom_{O_v^{ur}}(coker(\iota), \Bbb G_{m, O_v}) = O_v^{ur}[{\bf X}]^\times /{O_v^{ur}}^\times$$
as $\Gal(k_v^{ur}/k_v)$-module. Let $$ {\bf B} = \{ f\in O_v^{ur}[{\bf X}]^\times: \  f(1_{\bf T})=1 \} $$ which is stable under the action of $\Gal(k_v^{ur}/k_v)$. Then $O_v^{ur}[{\bf X}]^\times={O_v^{ur}}^\times \oplus {\bf B}$ as $\Gal(k_v^{ur}/k_v)$-module and $$coker(\iota) \times_{O_v} O_v^{ur} \cong \Spec(O_v^{ur}[{\bf B}])  \ \ \ \text{induced by} \ \ \ {\bf B} \cong O_v^{ur}[{\bf X}]^\times /{O_v^{ur}}^\times $$ is compatible with ${\Gal(k_v^{ur}/k_v)}$-action by Theorem 1.2 in Expos\'e VIII of \cite{SGA3}. Moreover, the natural inclusion of $O_v^{ur}$-algebras $O_v^{ur}[{\bf B}] \subseteq O_v^{ur}[{\bf X}]$ which is also compatible with ${\Gal(k_v^{ur}/k_v)}$-action gives the extension ${\bf X} \xrightarrow{\phi} coker(\iota)$ of ${\bf T}\xrightarrow{\phi} coker(\iota)$ over $O_v$.

Write $${\bf T}\times_{O_v} O_v^{ur} = \Spec( O_v^{ur}[x_1, x_1^{-1},\cdots, x_{s},x_{s}^{-1}, y_1, y_1^{-1}, \cdots, y_t, y_t^{-1}]) $$ and
 $$ {\bf X}\times_{O_v} O_v^{ur} = \Spec( O_v^{ur}[x_1, \cdots, x_s, y_1, y_1^{-1}, \cdots, y_t, y_t^{-1} ] ) $$ such that $x_i(1_{\bf T})=y_j(1_{\bf T})=1$ for $1\leq i\leq s$ and $1\leq j\leq t$ by the given diagram.  Then $$coker(\iota) \times_{O_v} O_v^{ur} = \Spec (O_v^{ur} [y_1, y_1^{-1}, \cdots, y_t, y_t^{-1}]) $$ and $$\phi^{ur}=\phi\times_{O_v} O_v^{ur}: \ \  {\bf X}\times_{O_v} O_v^{ur} \rightarrow coker(\iota) \times_{O_v} O_v^{ur}$$ is the projection and
 $$ \phi^{-1}(1)\times_{O_v} O_v^{ur} = (\phi^{ur})^{-1}(1) =  \Spec (O_v^{ur} [x_1, \cdots, x_{s}]) .$$ Since $$div_{({\bf X}\times_{O_v} O_v^{ur}) \setminus ({\bf T} \times_{O_v} O_v^{ur})}(x_i)=div_{{\bf X}\times_{O_v} O_v^{ur}}(x_i)$$ and the action of ${\Gal(k_v^{ur}/k_v)}$ on $\{div_{{\bf X}\times_{O_v} O_v^{ur}}(x_i)\}_{i=1}^s $ is same as the action on the coordinates $\{x_i\}_{i=1}^s$ by smoothness of $ {\bf X}\times_{O_v} O_v^{ur}$ and the normalization of $x_i$ for $1\leq i\leq s$, one concludes that
$$\phi^{-1}(1) = \prod_{i=1}^h \Res_{O_{k_i}/O_v} (\Bbb A_{O_{k_i}}^1) $$  as required.
\end{proof}

\begin{thm}\label{end} Any smooth toric variety satisfies strong approximation with Brauer-Manin obstruction off $\infty_k$. \end{thm}
\begin{proof} Let $(T\hookrightarrow X)$ be a smooth toric variety over $k$ and $\mathfrak F$ be the set of all open affine toric sub-varieties over $\bar k$. Since there are only finitely many $T(\bar k)$ orbits over $\bar k$ by Lemma \ref{orb}, one gets $\mathfrak F$ is finite. Moreover if $A$ and $B$ are in $\mathfrak F$, then $A\cap B\in \mathfrak F$ and $\sigma(A)\in \mathfrak F$ for any $\sigma\in \Gamma_k$ by the separateness of $X$ over $k$. Let $k'/k$ be a finite Galois extension such that $T\times_k k'\cong \Bbb G_m^n$ and $U$ is defined over $k'$ and $U\cong \Bbb A^{s_U}\times \Bbb G_m^{t_U} $ with non-negative integers $s_U$ and $t_U$ over $k'$ for all $U\in \mathfrak F$.

By Proposition \ref{red}, there is a unique open toric subvariety $Y\subset X$ of pure divisorial type over $k$ such that $dim(X\setminus Y)< \dim(T)-1$. Let $S$ be a finite subset of $\Omega_k$ containing $\infty_k$ and $\bf X$, $\bf Y$ and $\bf T$ be the integral model of $X$, $Y$ and $T$ over $O_{k,S}$ respectively such that

1) Every prime $v\not\in S$ is unramified in $k'/k$.

2) The open immersion $T\hookrightarrow X$ and the action $T\times_k X\xrightarrow{m_X} X$ extend to
$$ {\bf T}\hookrightarrow {\bf X} \ \ \ \text{and} \ \ \ {\bf T}\times_{O_{k,S}} {\bf X} \xrightarrow{m_{\bf X}} {\bf X} $$ over $O_{k,S}$.

3) The open immersion $T\hookrightarrow Y$ and the action $T\times_k Y\xrightarrow{m_Y} Y$ extend to
$$ {\bf T}\hookrightarrow {\bf Y} \ \ \ \text{and} \ \ \ {\bf T}\times_{O_{k,S}} {\bf Y} \xrightarrow{m_{\bf Y}} {\bf Y} $$ over $O_{k,S}$.

4) The open immersion $Y\hookrightarrow X$ extends to ${\bf Y} \hookrightarrow {\bf X}$ over $O_{k,S}$.

5) Let $\underline{ \mathfrak F}$ be the set of an integral model $\bf U$ over the integral closure $O_{k',S}$ of $O_{k,S}$ in $k'$ with an open immersion ${\bf U} \hookrightarrow {\bf X}\times_{O_{k,S}} O_{k',S}$ over $O_{k',S}$ which extends $U \hookrightarrow X\times_k k'$ over $k'$ for each element $U\in \mathfrak F$ such that $${\bf A} \cap {\bf B} \in \underline{\mathfrak F} \ \ \ \text{and} \ \ \ \sigma({\bf A})\in \underline{\mathfrak F}$$ whenever $\bf A, \bf B\in \underline{\mathfrak F}$ and $\sigma\in \Gal(k'/k)$. Moreover
$${\bf T}\times_{O_{k,S}} O_{k',S} \cong \Bbb G_{m, O_{k',S}}^n  \ \ \ \text{and} \ \ \ {\bf X}\times_{O_{k,S}} O_{k',S} = \bigcup_{{\bf U}\in \underline{\mathfrak F}} {\bf U}   $$ with ${\bf U} \cong \Bbb A_{O_{k',S}}^{s_U}\times \Bbb G_{m, O_{k',S}}^{t_U}$ over $O_{k',S}$ for each ${\bf U} \in \underline{\mathfrak F}$.

\medskip

Let $W=\prod_{v\in \Omega_k} W_v$ be an open subset of $X(\mathbf A_k)$ and $S_1$ be a finite subset of $\Omega_k$ containing $S$ such that
$$ (x_v)_{v\in \Omega_k} \in  W\cap X(\mathbf A_k)^{\Br_a X} \ \ \  \text{and} \ \ \  W_v={\bf X}(O_v)$$
for all $v\not\in S_1$.

For $v\in S_1$, we can assume that $x_v\in T(k_v)\cap W_v\subseteq Y(k_v)\cap W_v$ by Proposition \ref{inv}.

For $v\not\in S_1$, we can assume that $x_v\in {\bf T}(O_v)$. Indeed, since $x_v\in {\bf X}(O_v)$, there is ${\bf U}\in \underline{\mathfrak F}$ in the above condition 5) such that $x\in {\bf U}(O_{k'_w})$ for a prime $w|v$ in $k'$, where $O_{k'_w}$ is the ring of integers of $k'_w$. By the above condition 5)
$$  \bigcap_{\sigma\in \Gal(k'_w/k_v)} \sigma ({\bf U}) \in \underline{\mathfrak F}$$ and there is an affine scheme ${\bf M}_v$ over $O_v$ such that
 $$ {\bf M}_v \times_{O_v} O_{k'_w} =  (\bigcap_{\sigma\in \Gal(k'_w/k_v)} \sigma ({\bf U}))\times_{O_{k',S}} O_{k'_w} $$
 with $x\in {\bf M}_v(O_v)$ by Galois descent. By the above condition 1) and 5), one can apply Lemma \ref{integ} to $({\bf T}_v={\bf T}\times_{O_{k,S}} O_v \hookrightarrow {\bf M}_v)$ and obtain a surjective homomorphism of group schemes ${\bf T}_v \xrightarrow{\phi} {\bf T}_1$ for some commutative group scheme ${\bf T}_1$ over $O_v$ such that
$$ ker(\phi) = \prod_{i=1}^h \Res_{O_{k_i}/O_{v}}(\Bbb G_{m, O_{k_i}})$$ where $O_{k_i}$'s are the rings of integers of finite unramified extensions $k_i/k_v$ for $1\leq i\leq h$. Moreover, this map $\phi$ can be extended to a morphism ${\bf M}_v \xrightarrow{\phi} {\bf T}_1$.
Since $H^1_{et}(O_v, ker(\phi))=0$, one has ${\bf T}_v(O_v) \xrightarrow{\phi} {\bf T}_1(O_v)$ is surjective by \'etale cohomology.
If $x_v\not\in {\bf T}(O_v)$, there is $t_v\in {\bf T}_v(O_v)$ such that $\phi(x_v)=\phi(t_v)$. By Proposition \ref{str} or the proof of Proposition \ref{inv}, one has
$$\inv_v(\xi(x_v))=\inv_v(\xi(t_v))$$ for all $\xi\in \Br_1(X)$. Therefore one can replace $x_v$ with $t_v$ if necessary.

Therefore one can assume $$(x_v)_{v\in \Omega_k} \in [\prod_{v\in S_1} (W_v\cap Y(k_v))\times \prod_{v\not\in S_1} {\bf Y}(O_v)]\cap  Y(\mathbf A_k)^{\Br_a(Y)}$$ by the above condition 3) and $\Br_a(X)\cong \Br_a(Y)$ induced by open immersion.  By proposition \ref{typed}, there is $y\in Y(k)\subseteq X(k)$ such that $$y\in \prod_{v\in S_1} (W_v\cap Y(k_v))\times \prod_{v\not\in S_1} {\bf Y}(O_v) \subseteq \prod_{v\in \infty} X(k_v) \times \prod_{v\not\in \infty_k} W_v $$ by the above condition 4) as desired. \end{proof}

\section{An example}

At the end of \cite{HaVo}, Harari and Voloch constructed an open curve which does not satisfy strong approximation with Brauer-Manin obstruction. However their counter-example is not geometrically rational. Colliot-Th\'el\`ene and Wittenberg gave an open rational surface over $\Bbb Q$ (Example 5.10 in \cite{CTW})  which does not satisfy strong approximation with Brauer-Manin obstruction. Here we provide another such open rational surface. We explain that the complement of a point in a toric variety may no longer satisfy strong approximation with Brauer-Manin obstruction. We also show that strong approximation with Brauer-Manin obstruction is not stable under finite extensions of the ground field.

Before giving the explicit example, we have the following lemma.

\begin{lem}\label{fiber} Let $f: \ X\rightarrow Y$ be a morphism of schemes over a number field $k$ such that the induced map $f^*: \Br(Y) \rightarrow \Br(X)$ is surjective.
If $Y(k)$ is discrete in $Y(\mathbf A_k^S)$ and $X$ satisfies strong approximation with Brauer-Manin obstruction off $S$ for some finite subset $S$ of $\Omega_k$, then any fiber $f^{-1}(y)$ satisfies strong approximation off $S$ for $y\in Y(k)$.
\end{lem}
\begin{proof} Since $Y(k)$ is discrete in $Y(\mathbf A_k^S)$, there is an open subset $U_y$ of $Y(\mathbf A_k^S)$ such that $$Y(k)\cap U_y=\{y\}$$ for each $y\in Y(k)$. Let $$(x_v)_{v\not\in S}\in W \subseteq f^{-1}(y)(\mathbf A_k^S)$$ be a non-empty open subset. Since $f^{-1}(y)$ is a closed sub-scheme of $X$, there is an open subset $W_1$ of $X(\mathbf A_k^S)$ such that $W=W_1\cap [f^{-1}(y)(\mathbf A_k^S)]$. Let $x_v\in f^{-1}(y)(k_v)$ for $v\in S$. Then
$$ (x_v)_{v\in \Omega_k} \in [\prod_{v\in S} X(k_v) \times (W_1\cap f^{-1}(U_y))] \cap X(\mathbf A_k)^{\Br(X)} \neq \emptyset $$ by the surjection of $f^*: \Br(Y) \rightarrow \Br(X)$ and the functoriality of Brauer-Manin pairing. Since $X$ satisfies strong approximation with Brauer-Manin obstruction off $S$, there is $x\in X(k)$ such that $x\in W_1\cap f^{-1}(U_y)$. This implies that $f(x)\in U_y$ and $f(x)=y$. Therefore $x\in W$ as desired.\end{proof}

\begin{exa}\label{contre} Let $X=(\Bbb A^1\times_{k} \Bbb G_m) \setminus \{ (0,1) \} $ be a rational open surface over a number field $k$.

 1) If $k=\Bbb Q$ or an imaginary quadratic field, then $X$ does not satisfy strong approximation with Brauer-Manin obstruction off $\infty_k$.

 2) Otherwise $X$ satisfies strong approximation with Brauer-Manin obstruction off $\infty_k$.
\end{exa}
\begin{proof} 1) If $k=\Bbb Q$ or an imaginary quadratic field, one takes $Y=\Bbb G_m$ and the morphism $f: X\rightarrow Y$ by restriction of the projection map $\Bbb A^1\times_k \Bbb G_m \rightarrow \Bbb G_m$ to $X$. Since $O_k^\times$ is finite, one has $Y(k)$ is discrete in $Y(\mathbf A_{k}^\infty)$. The morphism $f$ induces an isomorphism $$f^*: \Br(Y)=\Br(\Bbb G_m) \xrightarrow{\cong} \Br(\Bbb A^1\times \Bbb G_m)=\Br(X). $$ Suppose $X$ satisfies strong approximation with Brauer-Manin obstruction off $\infty_k$. Then all fibers $f^{-1}(y)$ satisfy strong approximation off $\infty_k$ by Lemma \ref{fiber}. However $f^{-1}(1) \cong \Bbb G_m$ does not satisfy strong approximation off $\infty_k$. A contradiction is derived.

2)  Let $W=\prod_{v\in \Omega_k} W_v$ be an open subset in $X(\mathbf A_k)$ with $ (x_v)_{v\in \Omega_k} \in W\cap X(\mathbf A_k)^{\Br_1(X)}$. There is a finite subset $S$ of $\Omega_k$ containing $\infty_k$ such that
$$\begin{cases} x_v \in U_v\times V_v \subseteq W_v \subseteq (k_v^\times\times k_v\setminus \{(1,0)\})  \ \ \ & \text{for} \  v\in S \\
x_v\in W_v={\bf X}(O_v) = (O_v^\times\times O_v^\times) \cup ((O_v^\times\setminus (1+\pi_v O_v)) \times O_v) \ \ \  & \text{for} \  v\not\in S
\end{cases} $$ where $U_v$ and $V_v$ are the open subsets of $k_v^\times$ and $k_v$ respectively for $v\in S$ and $\pi_v$ is the uniformizer of $k_v$ for $v\not\in S$. Consider two projection
$$ p: \Bbb G_m \times_k \Bbb A^1 \rightarrow \Bbb G_m \ \ \ \text{and} \ \ \ q: \Bbb G_m \times_k \Bbb A^1 \rightarrow \Bbb A^1 . $$

If $k$ is neither $\Bbb Q$ nor an imaginary quadratic field, then $O_k^\times$ is infinite. Therefore $k^\times$ is not discrete in $\Bbb G_m(\mathbf A_k^\infty)$. Since $k^\times $ is dense in $\rm{Pr_\infty}(\Bbb G_m(\mathbf A_k)^{\Br_a(\Bbb G_m)})$, one concludes that $k^\times \setminus \{1\} $ is also dense in $\rm{Pr_\infty}(\Bbb G_m(\mathbf A_k)^{\Br_a(\Bbb G_m)})$. By the functoriality of Brauer-Manin pairing, one has $p((x_v))\in \Bbb G_m(\mathbf A_k)^{\Br_a(\Bbb G_m)}$. Choose an open subset $\prod_{v\in \Omega_k} M_v$ of $\Bbb G_m(\mathbf A_k)$ containing $p((x_v)_{v\in \Omega_k})$ such that
 $$ \begin{cases}
M_v=k_v^\times  \ \ \ & v\in \infty_k \\
 M_v=U_v \ \ \ & v\in S\setminus \infty_k \\
 M_v = O_v^\times \ \ \ & v\not\in S
 \end{cases} $$
There is $b\in k^\times \setminus \{1 \}$ such that $b\in \prod_{v\in \Omega_k} M_v$.  Let $S_1$ be a finite subset of $\Omega_k$ containing $S$ such that $b-1\in O_v^\times$ for all $v\not\in S_1$. Choose an open subset $\prod_{v\in \Omega_k} N_v$ of $\mathbf A_k$
 $$ \begin{cases}
N_v=k_v \ \ \ & v\in \infty_k \\
 N_v=V_v \ \ \ & v\in S \\
 N_v = O_v^\times \ \ \ & v\in S_1\setminus S \\
 N_v = O_v \ \ \ &  v\not\in S_1
 \end{cases} $$
Then there is $c\in k^\times$ such that $c\in \prod_{v\in \Omega_k} N_v$ by strong approximation for $\Bbb A^1$. Then $(b,c)\in W$ as desired.
\end{proof}

\bigskip

\noindent{\bf Acknowledgements.}
Special thanks are due to J.-L. Colliot-Th\'el\`ene who suggested significant improvement from the original version. We would like to thank Jiangxue Fang for helpful discussion on toric varieties and David Harari for drawing our attention to \cite{ChTs}. Part of the work was done when the second named author visited IHES from Nov.2013 to Dec.2013 and was also supported by NSFC grant no. 11031004.


\end{document}